\documentclass[12pt]{amsart}
\usepackage[centertags]{amsmath}
\usepackage{amsfonts}
\usepackage{amssymb}
\usepackage{amsthm}
\usepackage{newlfont}
\usepackage{bm}
\newlength{\defbaselineskip}
\newcommand{\setlinespacing}[1]%
{\setlength{\baselineskip}{#1 \defbaselineskip}}

\def\bbR{\mathbb{R}}

\def\bbN{\mathbb{N}}

\def\1{\mathbf{1}}

\def\bbR{\mathbb{R}}

\def\bbN{\mathbb{N}}

\theoremstyle{plain}
\newtheorem{thm}{Theorem}[section]
\newtheorem{cor}[thm]{Corollary}
\newtheorem{lem}[thm]{Lemma}
\newtheorem{prop}[thm]{Proposition}
\newtheorem{example}[thm]{Example}

\theoremstyle{definition}
\newtheorem{defn}{Definition}[section]
\def\dss{\displaystyle}
\theoremstyle{remark}

\numberwithin{equation}{section}

\begin{document}
	
	\def\dss{\displaystyle}
	
\title[  On Countably $\alpha$-Compact Topological Spaces]{On Countably $\alpha$-Compact Topological Spaces}

\author[Eman Almuhur and Muhammad Ahsan Khan]{Eman Almuhur $^1$ $\&$ Muhammad Ahsan Khan $^2$}

\address{$^{1}$ Department of Basic Science and Humanities, Faculty of Arts and Science, Applied Science Private University, Amman, Jordan, Email}
\email{e\_almuhur@asu.edu.jo}
\address{$^{2}$  Department of Mathematics, University of Kotli Azad Jammu $\&$ Kashmir, Kotli 11100, Azad Jammu $\&$ Kashmir, Pakistan.}
\email{ahsankhan388@hotmail.com}

	\begin{abstract}
		In this paper, some features of countably $\alpha$-compact topological spaces are presented and proven. The connection between countably $\alpha$-compact, Tychonoff, and $\alpha$-Hausdorff spaces is explained. The space is countably $\alpha$-compact space iff every locally finite family of non-empty subsets of such space is finite is demonstrated. The countably $\alpha$-compact space with weight greater than or equal to $\aleph_0$ is the $\alpha$-continuous image of a closed subspace of the cube $D^{\aleph_0}$ is discussed. The boundedness of $\alpha$-continuous functions mapping $\alpha$-compact spaces to other spaces is cleared. Moreover, the $\alpha$-continuous function mapping the space $X$ to the countably $\alpha$-compact space $Y$ is an $\alpha$-closed subset of $X\times Y$ is argued and proved. We explained that the $\alpha$-continuous functions mapping any topological space to a countably $\alpha$-compact space can be extended over its domain under some constraints. We claimed that the property of being $\alpha$-compact is countably $\alpha$-compact but the converse is not and the countable union of countably $\alpha$-compact subspaces of $X$ is also countably $\alpha$-compact. 	
	\end{abstract}
	\maketitle
	\section{Introduction}
	
	Topology is one of the most known mathematical disciplines because of its rich theoretical foundations and its many useful applications to science and engineering.
	Generalized open sets are an important part of general topology, and a large amount of research has concentrated on them. Extending open sets to explore various modified forms of continuity, separation axioms, and other ideas is an important topic in general topology and real analysis. Typically, the notions of $\alpha$-open sets, initiated by Njstad \cite{1} in 1965, and the generalized closed ($g$-closed) subset of a topological space presented by Levine \cite{2} in 1970, are the most well-known and also represent sources of inspiration and have been extensively studied in the literature since then. Since then, lots of mathematicians have focused on using $\alpha$-open sets and generalized closed sets to generalize different ideas in topology. Dunham \cite{3} in 1982 used $g$-closed subsets of $X$ to define a new closure operator and hence a new topological space $(X,\tau^*)$ where he transferred regularity conditions of a topological space $(X,\tau)$ to separation conditions of the new topological space $(X,\tau^*)$.  The paper is organized as follows: By means of section~2, we want to make sure that the reader has become familiar to basic notations and useful facts, needed in this area when we are going to start the main work in section~3. Section ~3 is an important section section of this paper which is concerned with the countably $\alpha$-compact Topological spaces.

	\section{Notations and preliminaries}
	
	In a topological space $(X,\tau)$, if $A$ is a subset of $X$ such that $A\subseteq Int(Cl(Int(A)))$, then $A$ is said to be $\alpha$-open. $A$ is $\alpha$-closed if $A\supseteq Cl(Int(Cl(A)))$. A subset $A$ of a topological space $(X,\tau)$ is called $g$-closed if $C(A)\subseteq O(A)\subseteq \tau$ and $A$ is $g$-closed if $C(A)\subseteq O(A)\subseteq \tau$ (see  \cite{3}).  Najastad \cite{1} proposed that each open set is an $\alpha$-open set. Dunham \cite{4} proved that the $g$-closedness property is closed under arbitrary union and a topological space is $T_{\frac{1}{2}}$ provided that every $g$-closed set is closed.       
	
	Moreover, a space $X$ is $T_{\frac{1}{2}}$ if and only if every $\{x\}$ such that $x\in X$ is either closed or open if and only if every subset of $X$ is the intersection of all open or closed sets containing it if and only if for each subset $A$ of $X$, $A^\prime$ is closed where $A^\prime$ is the set of all derived sets containing $A$. 
	Now, $A$ is $g$-open if $X-A$ is $g$-closed. Typically, $\alpha$-closure of a subset of a topological space $X$ is the intersection of each $\alpha$-closed set containing it and denoted by $Cl_\alpha(A)$ where $A\subseteq X$ whereas $\alpha$-interior of a subset is the union of each $\alpha$-open set contained in it and denoted by $Int_\alpha(B)$ where $B\subseteq X$  \cite{5}. The $g$-closure of a subset $D$ of $X$ denoted by $Cl^*(D)$ is the intersection of the $g$-closed sets containing $D$ \cite{6}. The family of $\alpha$-open subsets of $X$ is denoted by $\alpha O(X)$ and the family of $\alpha$-closed subsets of $X$ is denoted by $\alpha C(X)$. The family of all generalized open subsets of $X$ is denoted by $GO(X)$ and the family of all generalized closed subsets of $X$ is denoted by $GC(X)$. $O(X,x)$ denotes the set of all open subsets of $X$. A function $f:X\longrightarrow Y$ is $\alpha$-continuous if the inverse image of each open set in $Y$ is $\alpha$-open in $X$ \cite{3}. A function $f$ is $g$-continuous if the inverse image of each $g$-open set in $Y$ is $g$-open in $X$ \cite{7}. 
	
	Cs’asz’ar \cite{8} proposed the notion of generalized topological spaces in the twentieth century, which has been researched by lots of mathematicians all over the world. As a result, mathematicians adopted a new approach, seeking to extend topological concepts to this new domain. A $GT \mu$ (short for generalized topology $\mu$) is a collection of subsets of $X$ that is closed under arbitrary unions on a nonempty set $X$(\cite{9}). The elements of $\mu$ will be referred to as $\mu$-open sets, and a subset $A$ of $(X,\mu)$ will be referred to as $\mu$-closed if $X-A$ is $\mu$-open. Clearly, a subset $A$ of $(X,\mu)$ is $\mu$-open if and only if there exists a $\mu$-open neighborhood of $x$ that is contained in $A$, or alternatively, $A$ is the union of $\mu$-open subsets for each $x\in A$. $(X,\mu)$ will be referred to as generalized topological space (briefly $GTS$). A topological space $(X,\tau)$ is said to be $\alpha$-$T_1$ if for each distinct points $x$ and $y$ in $X$, there exist two $\alpha$-open subsets $u$ and $v$ such that $x\in u$ but not $y$ and $y\in v$ but not $x$ (\cite{10}). $(X,\tau)$ is termed $g$-$T_1$ if there exist two $g$-open subsets $u$ and $v$ such that $x\in u$ but not $y$ and $y\in v$ but not $x$ (\cite{11}). $(X,\tau)$ is called $\alpha$-$T_2$ if for each distinct points $x$ and $y$ in $X$, there exist two disjoint $\alpha$-open subsets $u$ and $v$ such that, $x\in u$ and $y\in v$ (\cite{12}). $(X,\tau)$ is said to $g$-$T_2$ if for each distinct points $x$ and $y$ in $X$, there exist two disjoint $g$-open subsets $u$ and $v$ such that $x\in u$ and $y\in v$ (\cite{13}).

	\section{Countably $\alpha$-Compact Spaces}
	\begin{defn}\cite{14}
		A topological space $(X,\tau)$ is said to be countably $\alpha-$ compact  if each countable set of open $\alpha$-compact subsets that covers $X$ has a finite subcover.
	\end{defn}
	\begin{lem}
		Every $\alpha$-compact is countably $\alpha$-compact and each countably $\alpha$-compact is countably compact but the convers is not true.
	\end{lem}
	\begin{example}
		If $X=\left\{1\right\}\bigcup\left\{\beta_i:i\in\bbN\right\}$ and $\tau=\left\{\phi,\left\{1\right\}\right\}$, then $X$ is countably compact but not a $\alpha$-countably compact. Another example of $\alpha$-compact which is not countably $\alpha$-compact is the space of all cardinal numbers $W_0$ since if it is embedded in the space $W=W_0\bigcup\left\{\omega_1\right\}$, it will not be a closed subspace. A topological space $(X,\tau)$ is $\alpha$-compact if and only if it is countably $\alpha$-compact with Lindelof property.
	\end{example}
	\begin{thm}
		In the topological space $(X,\tau)$, the countable union of countably $\alpha$-compact subspaces of $X$ is also countably $\alpha$-compact.
	\end{thm}
	\begin{proof}
		Suppose that $\widetilde{F}=\left\{F_\beta: F_\beta\quad \hbox{is}\quad \alpha-\hbox{closed},\beta\in\Gamma\right\}$ be a family of $\alpha$-closed subsets of $X$. Consider $\left\{W_\beta:\beta\in\Gamma\right\}$ be an a countable cover of $Y$ where $W_\beta$ is a $\alpha$-open subset of $Y$. Let $\bigcup_{\beta\in\Gamma}W_\beta$, then $Y$ is a closed subspace of $X$, there exists a countable subset $\left\{\beta_1,\beta_2,\beta_3,\cdots\right\}$ of $Y$, such that $\displaystyle\bigcup_{i=1}^\infty W_{\beta_i }$ is a countable subcover of $Y$, since $Y$ is countably $\alpha$-compact subspace of $X$. Now, $\left\{(X-F_\beta)\cap W_\beta:\beta\in\Gamma\right\}$ covers $X$, hence $X=\displaystyle \bigcup_{i=1}^{\infty} (X-F_{\beta_i})\cap W_{\beta_i }$ is a countable subcover of $X$. Thus $X$ is countably $\alpha$-compact. If $Y$ is a closed subspace of a topological space $(X,\tau)$, then each $\alpha$-subset of $X$ is a $\alpha$-subset of $Y$ provided that it is a subset of $Y$ basically.
	\end{proof}
	\begin{cor} \label{p}
		A topological space $(X,\tau)$ is countably $\alpha$-compact if and only if every countable family of $\alpha$-closed subsets that has the finite intersection property has non empty intersection. If and only if for each decreasing sequence of $\alpha$-closed subsets $\left\{ F_{i}\right\}_{i\in\bbN}$ of $X$, we have $\displaystyle\bigcap_{i\in\bbN}F_i\neq\emptyset$.
	\end{cor}
	
	\begin{thm}
		A topological space $(X,\tau)$ is countably $\alpha$-compact if and only if every locally finite family of non-empty subsets of $X$ is finite.
	\end{thm}
	\begin{proof}
		By contradiction, suppose that some locally finite families of non-empty subsets of $X$ are infinite. So, there exists a locally finite family $\left\{F_i:i\in\bbN\right\}$ of non-empty $\alpha$-closed subsets of $X$. By corollary \ref{p}, $K_j=\displaystyle\bigcup_{j\in\bbN}\overline{F_j}$ forms a decreasing sequence such that, $\displaystyle\bigcap_{j\in\bbN}K_j=\emptyset$. Thus $X$ is not countably $\alpha$-compact which is a contradiction.
	\end{proof}
	\begin{thm}
		The sum $\displaystyle\oplus_{\beta\in\Gamma}X_\beta$is countably $\alpha$-compact for all $X_\beta\neq\emptyset$  if and only if $X_\beta$ is countably $\alpha$-compact for all $\beta\in\Gamma$and $\Gamma$ is finite.
	\end{thm}
	\begin{proof}
		Suppose that the sum $\displaystyle\oplus_{\beta\in\Gamma}X_\beta$ is countably $\alpha$-compact for all $X_\beta\neq\emptyset$ , then all spaces $X_\beta$ are $\alpha$-compact because each one is an $\alpha$-closed subspace of $X$ and $\Gamma$ is finite. Conversely, if $\left\{X_j:j\in\bbN\right\}$ is a family of countably $\alpha$-compact spaces, then the result holds obviously.
	\end{proof}
	A function $f:X\longrightarrow Y$ is said to be $\alpha$-continuous if for all open subset $V\subseteq Y$, $f^{-1} (V)\subseteq X$ is open (\cite{15}). Every $\alpha$-continuous function is $\alpha$-continuous but the converse needs not to be true and every $g$-continuous function is $\alpha$-continuous but the converse needs not to be true. A function $f:X\longrightarrow Y$ is $\alpha$-compact if and only if for every $x\in X$ and open (resp. closed) subset $V\subseteq Y$ containing $f(x)$, there exists a $\alpha$-open (resp. closed) subset $U\subseteq X$ such that $f(U)\subseteq Y$. The composition of $\alpha$-continuous and a continuous function is $\alpha$-continuous but the composition of two $\alpha$-continuous functions needs not to be continuous. A function $f:X\longrightarrow Y$ is be $\alpha$-open if the image of every open (respectively closed) subset of $X$ is $\alpha$-open (respectively $\alpha$-closed) subset of $Y$ (\cite{16} ). A function $f:X\longrightarrow Y$ is $\alpha$-open if and only if for each neighborhood $W$ of a point $x$  in  $X$, there exists a $\alpha$-open subset $V$ in $Y$ such that containing $f(x)$ such that $V\subseteq f(W)$ (\cite{17}).
	
	\begin{prop}
		Let  $(X,\tau)$ and $(Y,\sigma)$ be the topological spaces, if $f:X\longrightarrow Y$ is an $\alpha$-continuous function, then:
		\begin{itemize}
			\item[\rm(i)]  If $X$ is a countably $\alpha$-compact space, then $f$ is bounded.

			\item[\rm(ii)] If $Y$ is $\alpha$-compact, then the graph of $f$ is an $\alpha$-closed subset of $X\times Y$.
		\end{itemize}
		
	\end{prop}	
	
	\begin{prop}
		For the topological spaces $(X,\tau)$, the following are equivalent:
		\begin{itemize}
			\item[\rm(i)]   Countably $\alpha$-compactness property is invariant under  $\alpha$-continuous functions.

			\item[\rm(ii)] Every countably $\alpha$-compact subspace of the $\bbR$ is $\alpha$-compact.
			
			\item[\rm(iii)] The class of countably $\alpha$-compact spaces is perfect.
		\end{itemize}
		
	\end{prop}	
	\begin{prop}
		The Cartesian product of countably $\alpha$-compact space $X$ and $\alpha$-compact space $Y$ is countably $\alpha$-compact. The Cartesian product of two countably $\alpha$-compact spaces needs not to be compact.
	\end{prop}
	
	Recall from \cite{18} that a topological space $(X,\tau)$ is $\alpha$-regular if for each point $x$ in $X$ and a closed subset $F$ of $X$ that does not contain $x$, there exist two disjoint open sets $U$ and $V$ such that $x$ belongs to $U$ and $F$ is a dense subset of $V$.
	\begin{defn}
		A topological space $(X,\tau)$ is completely $\alpha$-regular if points can be separated from closed sets via $\alpha$-continuous real-valued functions.
	\end{defn}
	\begin{prop}
		A topological space $(X,\tau)$ is Tychonoff if and only if it is completely $\alpha$-regular and $T_0$.
	\end{prop}
	
	\begin{prop}
		For the topological spaces $(X,\tau)$, the following are equivalent:
		\begin{itemize}
			\item[\rm(i)]   $\alpha$-$T_2$ countably $\alpha$-compact.

			\item[\rm(ii)] The projection function $p: X\times A (\aleph_0)\longrightarrow A(\aleph_0)$ is $\alpha$-closed.
			
			\item[\rm(iii)] Each $\alpha$-closed subspace of a countably $\alpha$-compact space is $\alpha$-compact.
		\end{itemize}
		
	\end{prop}
	\begin{prop}
		If $(X,\tau)$ is a countably $\alpha$ compact space, $(Y,\sigma)$ is an $\alpha-T_1$ space and $f:X\longrightarrow Y$ be $\alpha$-continuous function, then we have $f(\displaystyle \bigcap_{i\in\bbN}F_i)=\bigcap_{i\in\bbN}f(F_i)$ for any decreasing sequence of $\alpha$-closed subsets $F_1\supseteq F_2\supseteq\cdots$, of $X$. 
	\end{prop}
	\begin{example}
		The space $W$ that is obtained from the space of all cardinal numbers less than $\omega_1$ by making the set of all countable limits numbers $Z$ closed is non $\alpha$-regular countably $\alpha$-compact.
	\end{example}
	\begin{prop}
		For a countably $\alpha$-compact $(X,\tau)$, $|X|<e^{\chi(X)}$.
	\end{prop}
	\begin{cor}
		Each first countable countably $\alpha$-compact space has cardinality less than $\mathfrak{c}$.
	\end{cor}
	\begin{prop}
		Each countably $\alpha$-compact subspace of Sorgenfrey $\bbR_\ell$ line is countable.
	\end{prop}
	\begin{proof}
		Assume that $(X,\tau)$ is a countably $\alpha$-compact space and that $f:X \longrightarrow \bbR_\ell$ is a continuous injective function. The countablity of $X$ follows immediately by noting that  $f$ is a homeomorphism.
	\end{proof}
	\begin{thm}
		If $K$ is a subspace of the topological space $(X,\tau)$ and $(Y,\sigma)$ is an $\alpha$-compact space, then the $\alpha$-continuous function $g:X\longrightarrow Y$ has an extension $G$ over X if and only if for each two $\alpha$-closed subsets $F_1$ and $F_2$ of $Y$ , $g^{-1} (F_1 )\cap g^{-1} (F_2 )=\emptyset$.
	\end{thm}
	\begin{proof}
		Suppose that $G$ is an extension of the function $f$. Now $\overline{g^{-1}(F_1)}\cap\overline{g^{-1}(F_2)}\subseteq G^{-1}(F_1)\cap G^{-1}(F_2)=\emptyset$. On the other hand, if $\beta(x)$ is the set of neighborhoods of the point $x$ in $X$ and $\widetilde{U(x)}=\left\{\overline{g(K\cap U)}:  U\in \beta(x)\right\}$. Then  
		$\overline{g(K\cap U_1\cap U_2\cap \cdots U_n)}\subseteq\overline{g(K\cap U_1)}\cap \overline{g(K\cap U_2)}\cap \cdots \overline{g(K\cap U_n)}$ and $\beta(x)$ has the finite intersection property and $\cap\widetilde{U(x)}\neq\emptyset$ for all $x$. Let $y_1$ and $y_2$ be two distinct points in $\cap\widetilde{U(x)}$, then there exist two disjoint neighborhoods $V_1$ and $V_2$ of $y_1$ and $y_2$ respectively such that $\overline{V_1}\cap\overline{V_2}=\emptyset$, then $\overline{g^{-1}(V_1)}\cap\overline{g^{-1}(V_2)}=\emptyset$. Let $x\in X-\overline{g^{-1}(V_1)}$ and $y\in\widetilde{U(x)}\subseteq\overline{g(K-\overline{g^{-1}(V_1)})}$. Since $V_1\cap g(K-\overline{g^{-1}(V_1)})$ is empty and $y\in X-\overline{g(K-\overline{g^{-1}(V_1)})}$ which is a contradiction. Now $x\in\cap\widetilde{U(x)}$, the claim that the extension $G: X\longrightarrow Y$ is an $\alpha$-continuous. $G(x)=\cap\overline{g(K\cap U)}\subseteq V$, where $V$ is an $\alpha$-open subset of $Y$, hence there exists a family $\left\{U_1,U_2,\cdots U_n\right\}$ in $\beta(x)$ such that, $\displaystyle\bigcap_{i=1}^{n}\overline{g((K\cap U_i)}\subseteq V$, so $G(U)\subseteq V$.
	\end{proof}
	\begin{thm}
		The countably $\alpha$-compact space $(X,\tau)$ with weight greater than or equal to $\aleph_0$ is the continuous image of a closed subspace of the cube $D^{\aleph_0}$.
	\end{thm}
	\begin{proof}
		Suppose that the weight of the countably $\alpha$-compact space $Y$ is $\aleph_0$, $Y$ is homeomorphic to the Alexandroff cube $F^{\aleph_0}$, the identity mapping from $D^{\aleph_0}$ to $F^{\aleph_0}$ is continuous because each member of the canonical base of $F^{\aleph_0}$ is $\alpha$-open and $\alpha$-closed in $D^{\aleph_0}$. Now, if $F_1$ and $F_2$ are two disjoint $\alpha$-closed subsets of $Y$, then there exist two $\alpha$-closed subsets $K_1$ and $K_2$ of $F^{\aleph_0}$ such that $F_i=Y\cap K_i$ for all $i=1,2$. Now $Y\subseteq F^{\aleph_0}-(K_1\cap K_2)$. For every $x\in Y$, there exist $V_x\in\beta(x)$ such that $x\in V_{x}\subseteq F^{\aleph_0}-(K_1\cap K_2)$. Clearly, there exists a finite family $\left\{x_1,x_2,\cdots x_n\right\}$ in $Y$ such that $ Y\subseteq V_{x_1}\cup V_{x_2}\cup\cdots\cup V_{x_n}\subseteq F^{\aleph_0}-(K_1\cap K_2)$. Now $V=V_{x_1}\cup V_{x_2}\cup\cdots\cup V_{x_n}$ is an $\alpha$-open and $\alpha$-closed subset of $D^{\aleph_0}$ containing $K$ where $K=h^{-1}(Y)$ and $f=h|_K:K\longrightarrow Y$, hence $X=\overline{K}\subseteq V\subseteq D^{\aleph_0}-(K_1\cap K_2)$. Now $\overline{f^{-1}(K_1)}\cap\overline{f^{-1}(K_2)}\subseteq\overline{K}\cap K_1\cap K_2\subseteq\left(D^{\aleph_0}-(K_1\cap K_2)\right)\cup (K_1\cap K_2)=\emptyset$ because $f^{-1}(K_i)=\overline{K\cap K_i}\subseteq \overline{K}\cap K_i$ for every $i=1,2$. Thus, there exists an extension of the function $f$ which is $F:X \longrightarrow Y$ such that, $Y=f(K)\subseteq F(X)$, that is $Y$ is a continuous image of an $\alpha$-closed subspace of $D^{\aleph_0}$ in $X$.
	\end{proof}
	\begin{thm}
		The Cartesian product $\displaystyle\Pi_{\beta\in\Gamma}X_\beta$ is countably $\alpha$-compact iff for every $\beta\in\Gamma$ the spaces $X_\beta$ is countably $\alpha$-compact with $X_\beta\neq\emptyset$ for all $\beta\in\Gamma$.
	\end{thm}
	\begin{proof}
		The Cartesian product $\displaystyle\Pi_{\beta\in\Gamma}X_\beta$ is countably  $\alpha$-compact because the projection $p:X\longrightarrow X_\beta$ is an $\alpha$-continuous surjective function. On the other hand, if $\left\{X_\beta:\beta\in\Gamma\right\}$ is a family of non-empty countably $\alpha$-compact spaces, then $\displaystyle\Pi_{\beta\in\Gamma}X_\beta$ is an $\alpha-T_2$ space. Let $\tilde{F}=\left\{F_i: i\in\bbN\right\}$ of $\alpha$-closed subsets of $X$ that has the finite intersection property, so $\tilde{F}$ is contained in a maximal family $F_0$ that has the finite intersection property. Now the claim that $\cap\tilde{F}\neq\emptyset$. Since $F_0$ is maximal, $\displaystyle\cap F_i\in\tilde{F}$.  Consider the non-empty $\alpha$-open subset of $X$ say $U$, $U\cap F_i\neq \emptyset$ for all $i\in\bbN$.  The family $\left\{p_\beta(F_i):\beta\in\Gamma\right\}$ is an $\alpha$-closed subset of $X_\beta$ and has the finite intersection property. As a consequence, for all  $\beta\in\Gamma$, there exists a point $x_\beta\in\displaystyle\cap_{F_i\in\tilde{F}}\overline{P_\beta(F)}\subseteq X_\beta$. Now, if $x_\beta\in V_\beta$ for some $\alpha$-open subset $V_\beta$ of $X_\beta$, then  $V_\beta\cap p_\beta (F)\neq \emptyset$ for every $F\in\tilde{F}$ thus $p_\beta^{-1}(V_\beta)\in \tilde{F}$ and all members of the canonical bases of $x_\beta$ belong to $\tilde{F}$.
	\end{proof}
	\begin{cor}
		If $\displaystyle\Pi_{\beta\in\Gamma}X_\beta$ is an $\alpha$-continuous real valued function and $\left\{X_\beta: \beta\in \Gamma\right\}$ is a family of $\alpha$-compact spaces that depend on countably many coordinates where  $X_\beta$ is an $\alpha$-$T_2$  space and contains a dense subset $D$ that is represented as a countable union of $\alpha$-compact subsets, then $f$ depends on countably many coordinates.
	\end{cor}
	\begin{proof}
		Let $X_{\beta_1} \subseteq X_{\beta_2}\subseteq\cdots, $ be a sequence of $\alpha$-compact subsets of $X_\beta$, then $\displaystyle \bigcup_{i\in\bbN} X_{\beta_i}$ is dense in $X_\beta$ and $\Pi_{i\in\bbN}X_{\beta_i}=X_i$,  the result follows from previous theorem. 
		\begin{defn}
			A topological space $(X,\tau)$ is countably $\alpha$-pseudocompact if it is Tychonoff and any real valued $\alpha$-continuous defined on it is bounded.
		\end{defn}
	\end{proof}
	
	\begin{prop}
		If the topological space $(X,\tau)$ is Tychonoff, then the following are equivalent: 
		\begin{itemize}
			\item[\rm(i)]  $X$ is countably $\alpha$-pseudocompact.

			\item[\rm(ii)] Every locally finite family of non-empty $\alpha$-open subsets of $X$ is finite.
			
			\item[\rm(iii)] Every locally finite open cover of $X$ consisting of non-empty $\alpha$-open subsets is finite.
		\end{itemize}
		
	\end{prop}
	\begin{proof}
		We will prove only the implication $(i)\implies (ii)$ and leave others to readers as an easy exercise. Let $X$ be a countably $\alpha$-pseudocompact space, and there exists a locally finite family of non-empty $\alpha$-open subsets $\tilde{U}=\left\{u_i: i\in\bbN\right\}$. Consider $x_i\in u_i$, since $X$ is Tychonoff, there exists a real valued $\alpha$-continuous function $f_i: X\longrightarrow \bbR$ such that $f_i (x_i )=i$ and $f_i (X-u_i )\subseteq \left\{0\right\}$.  But $\tilde{U} $is locally finite, then $\displaystyle f(x)=\sum_{i=1}^{\infty}|f_i(x)|$ is a real valued $\alpha$-continuous function that is not bounded, thus $X$ is not $\alpha$-pseudocompact.
	\end{proof}
\begin{thm}
If the topological space $(X,\tau)$ is Tychonoff, then the following are equivalent: 
\begin{itemize}
	\item[\rm(i)] $X$ is countably $\alpha$-pseudocompact.
	\item[\rm(ii)] If $\tilde{V}=\left\{V_i:i\in\bbN\right\}$ is a family of non-empty $\alpha$-open subsets of $X$ with finite intersection property, then $\displaystyle\bigcap_{i\in\bbN} \overline{V_i}\neq\emptyset$.
\end{itemize} 
\end{thm} 
\begin{proof}
$(i)\implies (ii).$ Consider the countably $\alpha$-pseudocompact space $X$, if $U_1\supseteq U_2\supseteq \cdots,$ is a decreasing sequence of  non-empty $\alpha$-open subsets of $X$, then the family $\tilde{U}=\left\{U_i:i\in\bbN\right\}$ is not locally finite, so for some $x\in X$ such that its neighborhood meets infinitely many $U_i's$, thus $\displaystyle\bigcap_{i\in\bbN} \overline{V_i}\neq\emptyset$.
The implication $(ii)\implies (i)$ is immediate.
\end{proof} 
	\begin{prop}
		If $f:X\longrightarrow Y$ is a surjective $\alpha$-continuous function, $X$ is countably $\alpha$-pseudocompact space and $Y$ is Tychonoff, then $Y$ is countably $\alpha$-pseudocompact.
	\end{prop}
	\begin{prop}
		The Cartesian product of countably $\alpha$-pseudocompact space and $\alpha$-compact space is countably $\alpha$-pseudocompact.
	\end{prop}
	\begin{prop}
		The sum $\displaystyle\oplus_{\beta\in\Gamma}X_\beta$ is countably $\alpha$-pseudocompact for all $X_\beta\neq\emptyset$ if and only if $X_\beta$ is countably $\alpha$-pseudocompact for all $\beta\in\Gamma$ and $\Gamma$ is finite.  
	\end{prop}
	\begin{cor}
		Countably $\alpha$-pseudocompactness is not hereditary with respect to $\alpha$-closed sets.
	\end{cor}
	\begin{prop}
		Countably $\alpha$-pseudocompactness is not finitely multiplicative.
	\end{prop}
	\begin{prop}
		If the topological space $(X,\tau)$ is Tychonoff and $(Y,\sigma)$ is countably $\alpha$-pseudocompact and $f:X\longrightarrow Y$ is $\alpha$-continuous perfect function, then $X$ is countably $\alpha$-psuedocompact.  
	\end{prop}
	\section{Conclusion}
	The property of being $\alpha$-compact is countably $\alpha$-compact  but the converse is not and the countable union of countably $\alpha$-compact subspaces of $X$ is also countably $\alpha$-compact. Being countably $\alpha$-compact space is satisfied iff every locally finite family of non-empty subsets of such space is finite. The sum of countably $\alpha$-compact spaces for all non-empty countably $\alpha$-compact is countably $\alpha$-compact if and only if each is countably  $\alpha$-compact. Countably $\alpha$-compactness property is invariant under $\alpha$-continuous functions iff every countably $\alpha$-compact subspace of the real line $\bbR$ is $\alpha$-compact iff the class of countably $\alpha$-compact spaces is perfect. The countably $\alpha$-compact space with weight greater than or equal to $\aleph_0$ is the $\alpha$-continuous image of a closed subspace of the cube $D^{\aleph_0}$.
	\section{acknowledgments}
	My sincere thanks to Applied Science Private University in Amman, Jordan for providing a financial support to this research. 
	
\end{document}